\documentclass[a4paper, 11pt,reqno]{amsart}
\usepackage[english]{babel}
\usepackage[utf8]{inputenc}
\usepackage{amsthm, amsmath, amssymb, amsrefs, enumerate, enumitem, mathtools, nicefrac, bbm, mathrsfs}

\usepackage[top=1.1in, bottom=1.1in, left=0.95in, right=0.95in]{geometry}
\binoppenalty=\maxdimen
\relpenalty=\maxdimen

\usepackage[T1]{fontenc}
\usepackage{lmodern}

\usepackage[colorlinks=true,linkcolor=black,citecolor=black,filecolor=black]{hyperref}

\theoremstyle{plain}
\newtheorem{thm}{Theorem}[section]

\newtheorem{lemma}[thm]{Lemma}
\newtheorem{cor}[thm]{Corollary}
\newtheorem*{conj}{Conjecture}
\newtheorem{question}{Question}
\theoremstyle{definition}

\theoremstyle{remark}

\newtheorem{remark}[thm]{Remark}

\newcommand{\C}{\mathbb{C}}

\newcommand{\Z}{\mathbb{Z}}

\newcommand{\Q}{\mathbb{Q}}
\newcommand{\N}{\mathbb{N}}

\newcommand{\sptc}{\operatorname{spt-crank}}
\newcommand{\spt}{\operatorname{spt}}
\newcommand{\tcrk}{\operatorname{t-core-crank}}
\newcommand{\ork}{\operatorname{o-rank}}

\usepackage{placeins}

\usepackage{graphicx}
\graphicspath{ {./images/} }
\usepackage{caption}
\usepackage{subcaption}

\title{Equidistribution and partition polynomials}
\author{Amanda Folsom}
\address{Department of Mathematics and Statistics, Seeley Mudd Building, Amherst College, Amherst, MA, USA, 01002}
\email{afolsom@amherst.edu}

\author{Joshua Males}
\address{School of Mathematics, Fry Building, Woodland Road, Bristol, BS8 1UG}
\email{joshua.males@bristol.ac.uk}

\author{Larry Rolen}
\address{Department of Mathematics, 1420 Stevenson Center, Vanderbilt University, Nashville, TN 37240}
\email{larry.rolen@vanderbilt.edu}

\begin{document}

	\begin{abstract}
		Using  equidistribution criteria, we establish  divisibility by cyclotomic polynomials of several partition polynomials of interest, including $spt$-crank, overpartition pairs, and $t$-core partitions.  As corollaries, we obtain new proofs of various Ramanujan-type congruences for   associated partition functions.  Moreover, using results of Erd\"os and Tur\'an, we establish the equidistribution of roots of partition polynomials on the unit circle including those for the rank, crank, $spt$, and unimodal sequences.  Our results complement earlier work on this topic by Stanley, Boyer-Goh, and others. We explain how our methods may be used to establish similar results for other partition  polynomials of interest, and offer many related open questions and examples.
			\end{abstract}
	
	
	
	\thanks{The first author is  grateful for support from NSF Grant DMS-2200728. The research of the second author conducted for this paper is supported by the Pacific Institute for the Mathematical Sciences (PIMS). The research and findings may not reflect those of the Institute. This work was supported by a grant from the Simons Foundation (853830, LR). The third author is also grateful for support from a 2021-2023 Dean's Faculty Fellowship from Vanderbilt University.  The first author thanks the Max  Planck Institute for Mathematics, Bonn, Germany,  for its hospitality and support during portions of the writing of this paper.
	The authors also acknowledge the Vanderbilt University ``100 Years of Mock Theta Functions:
New Directions in Partitions, Modular Forms, and Mock Modular Forms'' Conference in May 2022, at which they discussed the topic of this paper and related ideas. This conference was supported by The Shanks Endowment,  Vanderbilt University, NSF Grant: “Conference on "100 Years of Mock Theta Functions; New Directions in Partitions, Modular Forms, and Mock Modular Forms"” award number: DMS-1951393, and NSA grant “100 Years of Mock Theta Functions,” award number: H98230-20-1-0022.
	}
	
	\maketitle

	\section{Introduction}\label{sec_intro} 
	 
	 	The theory of ranks and cranks was initiated in order to study congruences for the partition function $p(n)$, which counts the number of integer partitions of $n$, that is, the number of ways to write a non-negative integer $n$ as a non-increasing sum of positive integers (called ``parts'').  For example, $p(5)=7$, as $$5 = 4 + 1 = 3 + 2 = 3 + 1 + 1 = 2 + 2 + 1 = 2 + 1 + 1 + 1 = 1+ 1+ 1+1+1.$$  The definitions of the rank and crank of a partition $\lambda$ may appear artificial at first inspection:
	\begin{align*}   \text{rank}(\lambda) &:= \text{largest part of } \lambda - \text{number of parts of } \lambda,\end{align*}
	\begin{equation}\label{def_crank}
	\begin{aligned}\text{crank}(\lambda) &:= \begin{cases} \text{largest part of } \lambda & \text{if 1 is not a part of } \lambda, \\
	\mu(\lambda) - o(\lambda) & \text{if 1 is a part of } \lambda, \end{cases}
	\end{aligned}\end{equation}	
	where $\mu(\lambda)$ denotes the number of parts of $\lambda$ strictly larger than the number of $1$s in $\lambda$, and $o(\lambda)$ denotes the number of $1$s in $\lambda$.	However, these two perhaps seemingly peculiar partition statistics play significant roles in understanding partition numbers $p(n)$.    After Ramanujan conjectured his famous congruences modulo $5, 7$ and $11$ (for all $n\in\mathbb N_0$)
	\begin{equation}\label{eqn_ramcong} 
	\begin{aligned}
	p(5n+4)  \equiv 0 \pmod 5,   \  \\
	p(7n+5)  \equiv 0 \pmod 7,  \ \\
	p(11n+6)  \equiv 0 \pmod{11},
	\end{aligned}\end{equation}
	  while still an undergraduate in 1944, Dyson  \cite{Dyson1944} defined the rank of a partition in order to try and combinatorially explain them.   The following table illustrates how Dyson's rank divides the partitions of $5$ into $7$ groups of equal size:
	\begin{align*}
	\begin{array}{r|c|c}
	\text{partition} & \text{rank} & \text{rank} \pmod 7 \\ \hline
	5 & 5-1 & 4 \\
	4+1 & 4-2 & 2 \\
	3+2 & 3-2 & 1 \\
	3 + 1 + 1 & 3 - 3 & 0 \\
	2 + 2 + 1 & 2 - 3 & 6 \\
	2 + 1 + 1 + 1 & 2 - 4 & 5 \\
	1+1+1+1+1 & 1 - 5 &  \ 3.
	\end{array}
	\end{align*}
Further investigations led Dyson to conjecture that the rank would always divide the partitions of $7n+5$ (resp. $5n+4$) for any $n\in \mathbb N_0$ into 
	$7$ (resp. $5$) groups of equal size when reduced mod $7$ (resp. mod $5$), thereby explaining Ramanujan's congruences mod ${7}$ and mod ${5}$.  This was proved by Atkin and Swinnerton-Dyer in 1954 \cite{ASD}, and has led to further important related results in the literature.  A quick calculation reveals that the rank fails to explain Ramanujan's congruences mod $11$ in the same way as we now know it does for the moduli $5$ and $7$, and this led Dyson to conjecture the existence of another partition statistic  which would simultaneously explain all three of Ramanujan's congruences in \eqref{eqn_ramcong}.   Over four decades after Dyson's paper, important work of Garvan and Andrews \cite{AGBAMS, GarvanTAMS} led them to the definition of the crank  statistic in \eqref{def_crank}, resolving Dyson's question. 
		
	Also playing key roles in understanding partitions, ranks, and cranks, are their generating functions.  To describe this, we let $(a;q)_n \coloneqq \prod_{k=0}^{n-1} (1-aq^k)$ be the usual $q$-Pochhammer symbol $(n\in \mathbb N_0\cup \{\infty\})$. 
	%and also let %$(a;q)_\infty \coloneqq \lim_{n \to \infty} (a;q)_n$ and $(a,b;q)_n \coloneqq (a;q)_n (b;q)_n$.  
	Then we have that 
	\begin{align}\label{eqn: rank gen fn}
		\sum_{n \geq 0} \sum_{m \in \Z} N(m,n) w^m q^n = :   \sum_{n\geq 0} \operatorname{rank}_n(w) q^n = \sum_{n \geq 0} \frac{q^{n^2}}{(wq;q)_n (w^{-1}q;q)_n},	\end{align}
		and
	\begin{align}\label{eqn: crank gen fn}
	\sum_{n\geq 0} \sum_{m\in\mathbb Z} M(m,n) w^m q^n = : \sum_{n\geq 0} \operatorname{crank}_n(w) q^n = 
	\prod_{n\geq 1} \frac{(1-q^n)}{(1-wq^n)(1-w^{-1}q^n)}, 
	 \end{align}
	where $N(m,n)$ (resp. $M(m,n)$) denotes the number of partitions of $n$ with rank (resp. crank) $m$.  For example, when $w=1$ in \eqref{eqn: rank gen fn} we have that the partition generating function 
	$$\sum_{n\geq 0} p(n) q^n = \sum_{n \geq 0} \frac{q^{n^2}}{(q;q)^2_n} = \prod_{n\geq 1} \frac{1}{(1-q^n)}$$
	is essentially the reciprocal of the weight $1/2$ modular form $$\eta(\tau) = q^{\frac{1}{24}} \prod_{n\geq 1} (1-q^n),$$ with $q=e^{2\pi i \tau}, \tau \in \mathbb H$,  the upper-half of the complex plane. (Here we have also used Euler's product identity for the partition generating function, see e.g.\ \cite{AndrewsPartitions}.)  In general, when viewed as a two-variable function in $(z,\tau) \in \mathbb C\times \mathbb H$ (additionally letting $w=e^{2\pi i z}$),  the rank generating funciton in  \eqref{eqn: rank gen fn} does not posses the same strict modular properties as it does at $w=1$, but thanks to influential work of Zwegers~\cite{ZwegersThesis}, Bringmann-Ono~\cite{BringmannOno}, and Zagier~\cite{ZagierRanks},  we now know that (after a minor normalization) it is a mock Jacobi form.  The crank generating function in \eqref{eqn: crank gen fn} too possesses modular properties, and is (up to a minor normalization) a (true) Jacobi form.  Such connections between partition generating functions and modular(-type) forms have led to significant advances in our understanding in both the theory of partitions and in modular forms (broadly speaking).  For example, work of Hardy-Ramanujan and Rademacher established an exact formula for $p(n)$ using the modularity of its generating function;  we also now know to look at families of partition generating functions as potential explicit sources of holomorphic parts of harmonic Maass forms.  (For more on these topics, see e.g.\ \cite{BFOR}.)  
	
	When expanded as a $q$-series, the polynomial (in $w$) coefficients $\operatorname{rank}_n(w)$ and $\operatorname{crank}_n(w)$ of the generating functions \eqref{eqn: rank gen fn} and \eqref{eqn: crank gen fn} clearly carry important combinatorial information -- which it turns out may be revealed algebraically.  	To describe this, we review work of 
	Stanton, and Bringmann et al, and after Stanton define the modified rank and crank polynomials by
	\begin{align*} \operatorname{rank}_{\ell,n}^*(w) &:= 	\operatorname{rank}_{\ell n + \beta}(w) + w^{\ell n + \beta -2} - w^{\ell n + \beta - 1} +w^{2-\ell n -\beta} - w^{1-\ell n - \beta}, \\
	\operatorname{crank}_{\ell,n}^*(w) &:= 	\operatorname{crank}_{\ell n + \beta}(w) + w^{\ell n + \beta -\ell} - w^{\ell n + \beta } +w^{\ell-\ell n -\beta} - w^{-\ell n - \beta},
	\end{align*}
	where $\beta := \ell - (\ell^2-1)/24$.  In unpublished notes, Stanton made some related conjectures on divisibility by the cyclotomic polynomials $\Phi_\ell(w)$, including the following \cite{Stanton}.
	\begin{conj}[Stanton \cite{Stanton}]
	For $n\in \mathbb N_0$, the following are Laurent polynomials with non-negative coefficients:
	\begin{align*}
	\frac{\operatorname{rank}_{5,n}^*(w)}{\Phi_5(w)}, \frac{\operatorname{rank}_{7,n}^*(w)}{\Phi_7(w)}, 
	\end{align*}
	and
	\begin{align*}
	\frac{\operatorname{crank}_{5,n}^*(w)}{\Phi_5(w)}, \frac{\operatorname{crank}_{7,n}^*(w)}{\Phi_7(w)}, \frac{\operatorname{crank}_{11,n}^*(w)}{\Phi_{11}(w)}.
	\end{align*}
	\end{conj} 
	The connection between these conjectured -- and now proved  in \cite{BGRT,BMR} -- algebraic properties of polynomials and the combinatorial divisibilities prescribed by Ramanujan's congruences \eqref{eqn_ramcong} is given by Lemma \ref{Lem: equidist} \cite[Lemma 2.4]{BGRT}:  in short, divisibility by cyclotomic polynomials is equivalent to an equidistribution result on coefficients of (rank and crank) polynomials.    As explained in \cite{BGRT}, Stanton's modified rank and crank polynomials are designed to ensure  positivity, with the eventual goal of uncovering a combinatorial interpretation of what these positive coefficients count, thereby leading to a map between rank or crank congruence classes. Lemma \ref{Lem: pos coeffs}  \cite[Lemma 3.1]{BGRT} reveals that such positivity is related to unimodality of coefficients.
	
	In the sections that follow, using  equidistribution criteria, we establish  divisibility by cyclotomic polynomials of several partition polynomials of interest, including $spt$-crank, overpartition pairs, and $t$-core partitions (see Theorem \ref{thm: spt}, Theorem \ref{thm: ork3}, and Theorem \ref{thm: tcore}).  As corollaries, we obtain new proofs of various Ramanujan-type congruences for associated partition functions (see Corollary \ref{cor: spt57}, Corollary \ref{cor: pp3}, and Corollary \ref{cor: tcore}).  Moreover, using results of Erd\"os and Tur\'an, we establish the equidistribution of roots of partition polynomials on the unit circle including those for the rank, crank, $spt$, and unimodal sequences (see Theorem \ref{thm: rcequid}, Theorem \ref{thm: sptequid}, and Theorem \ref{thm: uniequid}). Our results complement earlier work on this topic by Stanley, Boyer-Goh \cite{BG1, BG}, and Boyer-Parry \cite{BP}.   We also explain how our methods may be used to establish similar results for other partition polynomials of interest, and offer many related open questions and examples in the ensuing sections.

\subsection*{Acknowledgements}
To be added.
We thank Ian Wagner, William Craig and Cristina Ballantine, and the referee for helpful comments on previous versions of the paper.
		
\section{Preliminaries}

We use the following lemmas of Bringmann--Gomez--Rolen--Tripp \cite{BGRT}. To state these, recall that the $\ell$th cyclotomic polynomial $\Phi_\ell(w)$   is the unique irreducible polynomial in $\mathbb Z[w]$ which divides $w^\ell -1$ but which does not divide $w^j-1$ for any $1\leq j < \ell$ ($\ell \in \mathbb N, \ell \geq 2$), and $\Phi_1(w) := w-1$.   We also define (for $(r,\ell) \in \mathbb Z \times \mathbb N$) 
\begin{align*}
		\widehat{f}_{r,\ell}  \coloneqq \sum_{j \equiv r \pmod{\ell}} [w^j] f(w).
\end{align*}  Here, by $[w^j] f(w)$ we mean the coefficient of $w^j$ in (the series expansion of) $f(w)$.
We also have the notion of divisibility in $\Q[w^{-1},w]$, where we say that $\Phi_\ell(w)$ divides $f(w)$ if there exists some $g(w) \in \Q[w^{-1},w]$ such that $f(w) = g(w) \Phi_\ell(w)$. The following lemma relates divisibility of polynomials to an equidistribution criterion over arithmetic progressions.

\begin{lemma}[Lemma 2.4 of \cite{BGRT}]\label{Lem: equidist}
	Let $f(w)$ be a Laurent polynomial in $\Q[w^{-1},w]$ and $\ell$ a prime. Then $\Phi_\ell(w)$ divides $f(w)$ in $\Q[w^{-1},w]$ if and only if 
	\begin{align*}
		\widehat{f}_{a,\ell} = \widehat{f}_{b,\ell}
	\end{align*}
for all $a,b$.
\end{lemma}
In particular, by showing divisibility by the $\ell$-th cyclotomic polynomial, we immediately obtain equidistribution modulo $\ell$ and vice-versa. The next result gives non-negativity of the coefficients of ${f(w)}/{\Phi(w)}$ under certain conditions.
\begin{lemma}[Lemma 3.1 of \cite{BGRT}]\label{Lem: pos coeffs}
	Let $f(w)$ be a symmetric unimodal Laurent polynomial that is divisible by $\Phi_\ell(w)$ for an odd prime $\ell$. Then the coefficients of $\frac{f(w)}{\Phi_\ell(w)}$ are non-negative. Moreover, if $f(w)$ is strictly unimodal then the coefficients of $\frac{f(w)}{\Phi_\ell(w)}$ are positive.
\end{lemma}

\section{Divisibility of (c)rank polynomials}
In the following three subsections, we give three examples of combinatorial objects that are well-studied in the literature. We establish divisibility properties of their two-variable generating functions, and establish new proofs of Ramanujan-type congruences. We also offer many related open questions of interest.

\subsection{The spt-crank}
In \cite{And}, Andrews introduced the function $\spt(n)$, which counts the number of smallest parts among the integer partitions of $n$.  For example,  the smallest parts among the partitions of three are underlined here: 
$$\underline{3}, 2+\underline{1}, \underline{1}+\underline{1}+\underline{1}$$ and hence we see that $\spt(3) = 5.$   Among the many interesting properties of $\spt$ now known to be true,
Andrews \cite{And} proved that the $\spt$ function satisfies three beautiful Ramanujan-type congruences \begin{align}\label{eqn: spt congruences}
\spt(5n+4) \equiv 0 \pmod{5}, \qquad \spt(7n+5) \equiv 0 \pmod{7}, \qquad \spt(13n+6) \equiv 0 \pmod{13},
\end{align} for $n\in \mathbb N_0$, 
 in analogy to the celebrated Ramanujan partition congruences for the partition function $p(n)$  modulo $5, 7,$ and $11$ in \eqref{eqn_ramcong}.  As the partition crank function was famously found to combinatorially explain Ramanujan's partition congruences (see Section \ref{sec_intro}),  it is natural   to search for an $\spt$-crank function that combinatorially explains Andrews's $\spt$-congruences above.   To this end, Andrews, Garvan and Liang \cite{AGL} defined an $\spt$-crank which explains Andrews' $\spt$-congruences modulo 5 and 7  in \eqref{eqn: spt congruences}.  
In order to prove their results, Andrews, Garvan, and Liang first introduced the set of vector partitions $V$, defined by the Cartesian product
\begin{align*}
	V = \mathcal{D} \times \mathcal{P} \times \mathcal{P},
\end{align*}
where $\mathcal{D}$ is the set of partitions into distinct parts and $\mathcal{P}$ is the set of partitions. Each vector partition $\overset{\rightarrow}{\pi} \in V$ comes equipped with a crank, defined by 
\begin{align*}
	\operatorname{crank}(\overset{\rightarrow}{\pi} ) = \#(\pi_2) -\#(\pi_3). 
\end{align*}
For a partition $\lambda$, let $s(\lambda)$ denote the smallest part, and define $s(-)=\infty$ for the empty partition. Then a central subset of $V$ in \cite{AGL} is given by
\begin{align*}
	S \coloneqq \{ \overset{\rightarrow}{\pi} = (\pi_1,\pi_2,\pi_3)  \in V \colon 1\leq s(\pi_1) < \infty \text{ and } s(\pi_1) \leq \min(s(\pi_2),s(\pi_3)) \}.
\end{align*}

For $\overset{\rightarrow}{\pi} \in S$, its weight is defined by $\omega_1(\overset{\rightarrow}{\pi}) \coloneqq (-1)^{\#(\pi_1) -1}$. Then the number of vector partitions of $n$ in $S$ with crank equal to $m$ counted in accordance with the weight $\omega_1$ is denoted by 
\begin{align*}
	N_S(m,n) \coloneqq \sum_{\substack{\overset{\rightarrow}{\pi} \in S, \lvert \overset{\rightarrow}{\pi} \rvert = n \\ \operatorname{crank}(\overset{\rightarrow}{\pi} )=m }} \omega_1(\overset{\rightarrow}{\pi}) .
\end{align*} 
Importantly, it turns out that
\begin{align*}
	\sum_{m \in \Z} N_S(m,n) = \spt(n).
\end{align*}

 Before stating our results,  define the $\spt$-crank polynomials $\sptc_n(w)$ ($n\in \mathbb N_0$) to be the $q$-series coefficients of the two-variable generating function for $N_S(m,n)$.  That is (see  \cite{AGL}),
\begin{align}\label{eqn: spt crank gen fn}
		\sum_{n \geq 1} \sum_{m \in \Z} N_S(m,n) w^m q^n =:  \sum_{n \geq 1} \sptc_{n}(w) q^n = \sum_{n \geq 1} \frac{q^n (q^{n+1};q)_\infty}{(wq^n;q)_\infty (w^{-1}q^n;q)_\infty}.
	\end{align}
Our first result, Theorem \ref{thm: spt} below, explains $\spt$-congruences another way, in terms of cyclotomic polynomial divisibility properties.  This leads to a new proof of Andrews' $\spt$-congruences modulo 5 and 7;  see Corollary \ref{cor: spt57}, its proof, and Remark \ref{rmk: corspt57}.
\begin{thm}\label{thm: spt}
We have that $\Phi_5(w)$ divides $\sptc_{5n+4}(w)$ and  $\Phi_7(w)$ divides $\sptc_{7n+5}(w) $ in $\mathbb Z[w,w^{-1}]$.

Moreover, the coefficients of both $\frac{ \sptc_{5n+4}(w)}{\Phi_5(w)}$ and $\frac{\sptc_{7n+5}(w)}{\Phi_7(w)}$ are non-negative.
\end{thm}
\begin{cor}\label{cor: spt57}  For $n\in \mathbb N_0$, we have that 
\begin{align*}
\spt(5n+4) \equiv 0 \pmod{5}, \qquad \spt(7n+5) \equiv 0 \pmod{7}.
\end{align*}  
\end{cor}
 
\begin{remark} \label{rmk: corspt57} Due to Theorem \ref{thm: spt} and Lemma \ref{Lem: equidist}, 
we in fact have an equidistribution result on the coefficients of the $\sptc_{5n+4}(w)$ and $\spt_{7n+5}(w)$ polynomials. While this implies   Corollary \ref{cor: spt57}, we establish this result more directly from Theorem \ref{thm: spt} below for simplicity.
\end{remark}

\begin{proof}[Proof of Theorem \ref{thm: spt}]
The divisibility of the relevant polynomials follows from \eqref{eqn: spt crank gen fn} and Lemma \ref{Lem: equidist},   using the fact that the $\sptc$ is equidistributed in these cases (see \cite{AGL}). To see that the coefficients of each of $\frac{ \sptc_{5n+4}(w)}{\Phi_5(w)}$ and $\frac{\sptc_{7n+5}(w)}{\Phi_7(w)}$ are non-negative it is enough to note that $N_S(m,n)$ is symmetric and unimodal in $m$ for each fixed $n$ by results of Chen--Ji--Zang \cite{CJZ}. Applying Lemma \ref{Lem: pos coeffs} finishes the proof.
\end{proof}

\begin{proof}[Proof of Corollary \ref{cor: spt57}]
By Theorem \ref{thm: spt} with $w=1$, we have that $\Phi_5(1)$ divides $\sptc_{5n+4}(1)$ and $\Phi_7(1)$ divides $\sptc_{7n+5}(1)$.  By  
 \eqref{eqn: spt crank gen fn}, we have that for $n\in \mathbb N_0$, $\sptc_n(1) = \spt(n)$, and by definition, we have that $\Phi_5(1) = 5$ and $\Phi_7(1) = 7$.  The corollary now follows.
\end{proof}

 Of course, since Lemma \ref{Lem: equidist} is an if-and-only-if statement, one could also prove the first part of the result in the reverse direction, by showing that the relevant cyclotomic polynomial divides the generating function on the given arithmetic progression. 
 
 \begin{question}\label{Q sptc}
 	Is there a simple way to show that the relevant cyclotomic polynomial divides the $\sptc$ on the given arithmetic progressions in Theorem \ref{thm: spt} without using Lemma \ref{Lem: equidist} and unimodality?
 \end{question}

\begin{remark}\label{Remark: sptc}
	In \cite[Theorem 4.1]{AGL}, Andrews, Garvan, and Liang show that the generating function for the $\sptc$ can be written (up to a multiplicative factor) as the difference between the ordinary -- unmodified -- crank and rank generating functions. 
	This may provide a starting point towards answering Question \ref{Q sptc}.

\end{remark}

While the non-negativity of the coefficients after dividing by the relevant cyclotomic polynomial is an easy consequence of Lemma \ref{Lem: pos coeffs}, the coefficients themselves remain mysterious. It would be extremely instructive to determine a combinatorial description for them.

\begin{question}\label{Q comb interp}
	Do the coefficients of $\frac{ \sptc_{5n+4}(w)}{\Phi_5(w)}$ and $\frac{\sptc_{7n+5}(w)}{\Phi_7(w)}$ have a combinatorial interpretation?
\end{question} Moreover, Theorem \ref{thm: spt} only deals with the cases of arithmetic progressions modulo $5$ and $7$, since here  the $\sptc$ explains the corresponding Ramanujan-type congruences. It is well-known that the $\sptc$ does not combinatorially explain the congruence $\spt(13n+6) \equiv 0 \pmod{13}$.
\begin{question}
Can one determine a combinatorial explanation for the $\spt$ congruence modulo $13$ -- and by establishing results as in this paper modulo $5$ and $7$?
\end{question}

Partition and smallest parts congruences modulo other primes and residue classes are of interest, including modulo the smallest primes $2$ and $3$;  see for example \cite{FKO},  or  \cite{OnoSubbarao, Radu} by Ono and Radu, which notably prove Subbarao's conjecture on the partition function modulo $2$ and $3$, or \cite{FO} on $\spt$-congruences modulo $2$ and $3$.  %

\begin{question}
	Is the $\sptc$ equidistributed on certain arithmetic progressions modulo $2$ and $3$?
\end{question}
\begin{remark}
 		While the $\spt$ function is known to be almost always even (see \cite[Theorem 1.3]{AGL2}), the possible Ramanujan-type congruence modulo $2$ may not hold with the exceptions captured by those terms given in \cite[Theorem 1.3]{AGL2}.
 \end{remark}

As noted in Section \ref{sec_intro}, it is well-known that the partition generating function is essentially a modular form.  Since the time of Hardy, Ramanujan, and Rademacher and their influential related work, mathematicians have produced a wealth of literature on modularity and partition functions -- broadly construed to also include Maass forms, mock theta functions, quantum modular forms, and other modular-type functions.  For example, in \cite{Bri}, Bringmann constructs a harmonic Maass form associated to the $\spt$-generating function, and  it is of interest to investigate  and establish   the modular properties of further $\spt$-generating functions as well as its applications. As explained in Remark \ref{Remark: sptc}, the generating function for $\sptc$ is essentially a difference between the rank and crank generating functions, thereby implying mock Jacobi properties.
 
\begin{question}
	What are the precise modular properties of the $\sptc$ generating function given in \eqref{eqn: spt crank gen fn} and can they be used to determine further congruences and asymptotic properties of the $\sptc$?
\end{question}

\subsection{The rank of overpartition pairs}
An \emph{overpartition} is a partition in which the first occurrence of a part may be overlined.  For example, there are four overpartitions of $2$, namely $1+1, \overline{1}+1, 2,$ and $\overline{2}$.  Corteel and Lovejoy helped pioneer the modern-day study of overpartitions and established several important results in \cite{CL}, noting that related combinatorial and $q$-hypergeometric results trace back to older work of MacMahon \cite{MacMahon},   Joichi-Stanton \cite{JoichiStanton}, and others. It is known that there are no Ramanujan-type congruences for overpartitions, see \cite[Theorem 1.2]{Dewar}, but for \emph{overpartition pairs}, they do exist.  To explain this, an overpartition pair of $n$ is a pair of overpartitions $(\mu,\lambda)$ where the sum of all of the parts is equal to $n$.   For example, there are 12 overpartition pairs of $n=2$ (noting that the definition allows the empty overpartition to be used for $\mu$ or $\lambda$).   Overpartition pairs have been of importance as associated to Ramanujan's ${}_1\psi_1$ summation, the $q$-Gauss identity, and other $q$-hypergeometric series \cite{LovOver}.    In \cite{BLOver}, Bringmann and Lovejoy establish the following overpartition pair congruence for  the overparition pair function $\overline{pp}(n)$ which the number of overpartition pairs of $n$ $(n\in\mathbb N_0)$:
\begin{align}\label{eqn: ppmod3}
	\overline{pp}(3n+2) \equiv 0 \pmod{3}.  
\end{align}
In the same way that the partition rank function can be used to combinatorially explain Ramanujan's partition congruences modulo 5 and 7, Bringmann and Lovejoy \cite[Theorem 1.2]{BL} use the overpartition pair rank function to explain \eqref{eqn: ppmod3} by splitting overpartition pairs into three equinumerous classes sorted by ranks.  To define this rank function, we let $\ell(\cdot)$ denote the largest part of a partition, and let  $n(\cdot)$ denote the number of parts;  their overlined counterparts count only parts which are overlined. With this, the rank of an overpartition pair $(\lambda,\mu)$  is defined by
 \begin{align*}
	\ell((\lambda,\mu)) - n(\lambda) - \overline{n}(\mu) - \chi((\lambda,\mu)),
\end{align*}
where $ \chi((\lambda,\mu))$ is defined to be $1$ if the largest part of $(\lambda,\mu)$ is non-overlined and in $\mu$, and $0$ otherwise.

Our next set of results are parallel to Theorem \ref{thm: spt} and Corollary \ref{cor: spt57}.  To state them, we 
 define the overpartition pair rank polynomials $\ork_n(w)$ ($n\in \mathbb N_0$) to be the $q$-series coefficients of the two-variable generating function for   $\overline{NN}(m,n)$, the number of overpartition pairs of $n$ with rank $m$.  Explicitly  \cite[Proposition 2.1]{BL}, we have that
\begin{align}\label{eqn: ork gen fn}
	 \sum_{\substack{n\geq0 \\ m \in \Z}} \overline{NN}(m,n) w^m q^n =: \sum_{n\geq0} \ork_{n}(w) q^n	 = \sum_{n \geq 0} \frac{(-1;q)_n^2 q^n }{(wq;q)_n(w^{-1}q;q)_n}.
\end{align}
 
Our first result on overpartition pair rank polynomials is the following.
\begin{thm}\label{thm: ork3}
	We have that $\Phi_3(w)$ divides $\ork_{3n+2}(w)$ in $\Z[w,w^{-1}]$.
\end{thm} 

\begin{proof}
	By Lemma \ref{Lem: equidist} it is enough to have that the $\ork$ is equidistributed on the arithmetic progression $3n+2$. This is shown by Bringmann and Lovejoy \cite[Theorem 1.2]{BL} who showed that the overpartition pair rank splits overpartition pairs into three equinumerous classes.
\end{proof}

As a corollary, we obtain a new proof of the overpartition rank congruence modulo 3 of Bringmann and Lovejoy discussed above.
\begin{cor}\label{cor: pp3}  For $n\in \mathbb N_0$, we have that 
\begin{align*}
\overline{pp}(3n+2) \equiv 0 \pmod{3}.
\end{align*}  
\end{cor}

\begin{remark} \label{rmk: corpp3} The contents of Remark \ref{rmk: corspt57} also apply here to $\operatorname{o-rank}_{3n+2}(w)$ and $\overline{pp}(3n+2)\pmod{3}$ in a similar manner.  \end{remark}

\begin{proof}[Proof of Corollary \ref{cor: pp3}]
By Theorem \ref{thm: ork3} with $w=1$, we have that $\Phi_3(1)$ divides $\ork_{3n+2}(1)$.  By  
 \eqref{eqn: ork gen fn}, we have that for $n\in\mathbb N_0$, $\ork_n(1) = \overline{pp}(n)$, and by definition, we have that $\Phi_3(1) = 3$.  The corollary now follows.
\end{proof}

Similar to the case of the $\sptc$, the (non-)unimodality of $\overline{NN}(m,n)$ in the $m$-aspect appears to be unknown in the literature. It is clear that $\overline{NN}(-m,n) = \overline{NN}(m,n)$, and numerical tests (using SageMath \cite{sage}) suggest that $\overline{NN}(m,n)$ is not unimodal in $m$, although this may be a situation that parallels the ordinary partition rank studied in \cite{BGRT}, where the authors introduced a slightly modified definition to ensure unimodality.

\begin{question}
	What are the (non-)unimodality properties of $\overline{NN}(m,n)$ in the $m$-aspect?
\end{question}

A choice of crank to explain a particular partition congruence is not unique. For example, very recently Wagner \cite{Wag} described a vast array of cranks for various partition-theoretic congruence families. In particular, a new crank statistic that explains the congruence $\overline{pp}(3n+2)\equiv 0\pmod{3}$ was introduced on p24 of \cite{Wag}, given by the pleasing infinite product formula $$\overline{C}_2(w;q) \coloneqq \prod_{n \geq 1} \frac{(1+w q^n)(1+w^{-1}q^n)}{(1-wq^n)(1+w^{-1}q^n)}.$$ One then obtains the analogues of Theorem \ref{thm: ork3} and Corollary \ref{cor: pp3} for this new crank. Moreover, this new crank numerically appears to unimodal, in contrast to the crank of Bringmann--Lovejoy. Unimodality would then lead to the non-negativity of coefficients of this new crank polynomial divided by $\Phi_{3n+2}(w)$.  One could ask for a combinatorial interpretation in analogy to Question \ref{Q comb interp}.
	
	\begin{question}
		What are the unimodality properties of $\overline{C}_2(w;q)$? Can one extend the ideas presented here to the wide class of crank functions given by Wagner in \cite{Wag}?
	\end{question}

\subsection{$t$-core partitions}\label{Sec: t-core}
Each partition comes equipped with a statistic called the \emph{hook length} of the partition. To explain this, we recall that every partition $\lambda = (\lambda_1,\lambda_2,\dots,\lambda_s)$ has a {\it Ferrers--Young diagram}
$$
\begin{matrix}
	\bullet & \bullet & \bullet & \dots & \bullet & \leftarrow & \lambda_1 \text{ many nodes}\\
	\bullet & \bullet & \dots & \bullet & & \leftarrow & \lambda_2 \text{ many nodes}\\
	\vdots & \vdots & \vdots & &  &  \\
	\bullet & \dots & \bullet & & & \leftarrow & \lambda_m \text{ many nodes},
\end{matrix}
$$
and each node has an associated {\it hook length}. The node in row $k$ and column $j$ has hook length given by
$h(k,j):=(\lambda_k-k)+(\lambda'_j-j)+1,$ where $\lambda'_j$ is the number of nodes in column $j$. This counts the number of nodes in the diagram directly below the given node plus the number to the right plus one (to count the given node itself). 
These numbers play many significant roles in combinatorics, number theory, and representation theory. For example, due to the Frame--Robinson--Thrall hook length formula \cite[6.1.19]{JK} we know that the counts of hook lengths of partitions control the number of irreducible representations of the symmetric groups $A_n$ and $S_n$.

For positive integers $t$, $t$-core partitions have also been of importance in combinatorial number theory.  A $t$-core partition of a positive integer $n$ is a partition of $n$ for which none of its hook lengths are divisible by $t$.  Let $c_t(n)$ denote the number of $t$-core partitions of $n$.  Granville and Ono \cite{GranvilleOno} showed that there are infinite families of congruences that $5$-core, $7$-core, and $11$-core partitions satisfy, including the Ramanujan-type congruences $(n\in \mathbb N_0)$
\begin{align}\label{eqn: tcorecong}
	c_5(5n+4)\equiv 0 \pmod{5}, \qquad c_7(7n+5) \equiv 0 \pmod{7}, \qquad c_{11}(11n+6) \equiv 0 \pmod{11}.
\end{align}

In order to combinatorially explain such congruences, Garvan, Kim, and Stanton \cite{GKS} introduced the $t$-core crank. To define it, we recall Bijection 2 of \cite{GKS}, which states that there is a bijection $\phi_2 \colon P_{\text{t-core}} \to \{ \overset{\rightarrow}{n} = (n_0,n_1,\dots,n_{t-1}) \colon n_i \in \Z, n_0 + n_1+\dots+n_{t-1}=0\}$, where
\begin{align*}
	\lvert \tilde{\lambda} \rvert = \frac{t \lvert\lvert  \overset{\rightarrow}{n}\rvert\rvert^2}{2} + \overset{\rightarrow}{b} \cdot \overset{\rightarrow}{n}, \qquad \overset{\rightarrow}{b} \coloneqq (0,1,\dots,t-1).
\end{align*}

Then the $t$-core crank definition is given algorithmically for a given partition $\lambda$. Choosing $t=5,7,11$ one begins by finding the $t$-core $\tilde{\lambda}$ of $\lambda$. Next, find $\phi_2(\tilde{\lambda}) = \overset{\rightarrow}{n}$. Then the $t$-core crank is given by the following mod $t$ combination
	\begin{align*}
		 4n_0 +n_1+n_3+4n_4 &\qquad \text{ for } t=5,\\
		 4n_0 +2n_1 + n_2+n_4+2n_5+4n_6  &\qquad \text{ for } t=7,\\
		4n_0 +9n_1 + 5n_2+3n_3 + n_4+n_6+3n_7+5n_8+9n_9+4n_{10}  &\qquad \text{ for } t=11.
	\end{align*}

Let $c_t(m,n)$ denote the number of $t$-core partitions of $n$ with $t$-core crank $m$. Then the $t$-core crank polynomials $\tcrk_n(w)$ are defined for $n \in \mathbb N_0$ by
\begin{align} \label{eqn: tcrk gen fn}
\sum_{\substack{n\geq0 \\ m \in \Z}} c_t(m,n) w^m q^n \eqqcolon  \sum_{n \geq 0} \tcrk_n(w) q^n.
\end{align}

Parallel to Theorem \ref{thm: spt}, but for all three moduli $5, 7$ and $11$, we establish the following result.
\begin{thm}\label{thm: tcore}
	We have that $\Phi_5(w)$ divides $\tcrk_{5n+4}(w)$,  $\Phi_{7} (w)$ divides $\tcrk_{7n+5}(w)$, and $\Phi_{11} (w)$ divides $\tcrk_{11n+6}(w)$ in $\Z[w,w^{-1}]$.
\end{thm}
\begin{proof}
	Garvan, Kim, and Stanton \cite{GKS} showed that the $\tcrk$ splits $c_5(5n+4)$, $c_7(7n+5)$, and $c_{11}(11n+6)$ into equinumerous classes (see also \cite[Theorem 3.1]{Garvan}). Combined with Lemma \ref{Lem: equidist}, one immediately obtains the result.
\end{proof}

As a corollary, we obtain a new proof of the Ramanujan-type congruences for $c_t(n)$ modulo $5, 7$ and $11$ in \eqref{eqn: tcorecong}.  
\begin{cor}\label{cor: tcore}  For $n\in \mathbb N_0$, we have that 
\begin{align*}
c_5(5n+4)\equiv 0 \pmod{5}, \qquad c_7(7n+5) \equiv 0 \pmod{7}, \qquad c_{11}(11n+6) \equiv 0 \pmod{11}.
\end{align*}  
\end{cor}

\begin{remark} \label{rmk: cortcore} The contents of Remark \ref{rmk: corspt57} also apply here to $\operatorname{t-core-crank}_{5n+4}(w),$  $\operatorname{t-core-crank}_{7n+5}(w),$ $\operatorname{t-core-crank}_{11n+6}(w)$, and $c_5(5n+4), c_7(7n+5)$ and $c_{11}(11n+6)$ $\pmod{5, 7,11}$   (respectively) in a similar manner.  \end{remark}

\begin{proof}[Proof of Corollary \ref{cor: tcore}]
By Theorem \ref{thm: tcore} with $w=1$, we have that $\Phi_5(1)$ divides $\tcrk_{5n+4}(1)$, that  $\Phi_7(1)$ divides $\tcrk_{7n+5}(1)$, and that $\Phi_{11}(1)$ divides $\tcrk_{11n+6}(1)$.  By  
 \eqref{eqn: tcrk gen fn}, we have that for $n\in \mathbb N_0$, $\tcrk_n(1) = c_t(n)$, and by definition, we have that $\Phi_5(1) = 5$,  $\Phi_7(1) = 7$, and  $\Phi_{11}(1) = 11$.  The corollary now follows.
\end{proof}

To the best of the authors' knowledge, the (non-)unimodality properties of $c_t(m,n)$ in the $m$-aspect are unknown in the literature.

\begin{question}\label{Q8}
	What are the (non-)unimodality properties of $c_t(m,n)$ in the $m$-aspect?
\end{question}

To obtain numerical evidence for this question, it is natural to ask for a $q$-series representation for the two-variable generating function given in \eqref{eqn: tcrk gen fn}. To the best of the authors' knowledge, this is unknown in the literature and so we leave this as a question.
\begin{question}
	Is there a (nice) $q$-hypergeometric or product expression (e.g. as in \eqref{eqn: rank gen fn}, \eqref{eqn: crank gen fn}, \eqref{eqn: spt crank gen fn}, \eqref{eqn: ork gen fn}) that allows one to answer (or make progress in answering) Question \ref{Q8}?
\end{question}

We end this section by noting that there are a plethora of other partitions statistics in the literature that one could ask similar questions for. For example, spt residual cranks of overpartitions, unimodal sequences, the rank of overpartition pairs and more general partition pairs \cite{Toh}, and the many crank/rank functions of Garvan and Jennings-Shaffer \cite{GJ,GJ2,JenningsShaffer,JS,JS2}.

\section{Principal Polynomial roots}
Refined information on partition statistics can be found by inspecting their so-called principal polynomials (defined explicitly in our cases below). This story begins with Stanley, who investigated the zeros of the partition polynomials
\begin{align*}
	F_n(w)\coloneqq \sum_{k=1}^{n} p_k(n) w^n,
\end{align*}
where $p_k(n)$ denotes the number of partitions into exactly $k$ parts. Stanley plotted the zeros of $F_{200}(w)$ and asked for their limiting behaviour as $n$ tends to $\infty$. This was settled in two beautiful papers of Boyer and Goh \cite{BG1,BG}, who discussed in detail the (rather exotic) zero-attractor of $F_n(x)$ (we refer the reader to these papers for more history and background on this topic). Boyer-Goh also prove similar results for the zero-attractors of several other common objects, including Appell and Euler polynomials \cite{BG-Appell,BG-Euler}, and Boyer--Parry proved similar results for traces of plane partitions \cite{BP}. In each case, the resulting zero-attractor is rather complicated, and the proofs require technical asymptotics and bounds. 

However, in some particular cases the roots of polynomials associated to partition-theoretic objects take a very simple shape - they become equidistributed on the unit circle as $n$ grows. To state this formally, we require classical results of Erd\"{o}s and Tur\'{a}n, which were conveniently packaged together in our setting by Granville in \cite{Granville}. Let $f(w) = \sum_{j=0}^{d} a_jw^j$ and  
\begin{align*}
	L(f) \coloneqq \frac{\sum_{j=0}^{d} |a_j|}{(|a_0||a_d|)^{\frac{1}{2}}}.
\end{align*}
Let $\nu_{ \{|z|=1\} }$ be the Haar measure on the unit circle (that is, equidistribution), and for a polynomial with not necessarily distinct roots $z_1,\dots z_d$ let $\mu_{ \{f\} }= \frac{1}{d}\sum_{j=1}^{d}\delta_{z_j}$ where $\delta_z$ is the Dirac delta measure. Then Theorem 1.3 of \cite{Granville} reads as follows. 
\begin{thm}\label{Thm: Granville}
	Suppose that $f_1,f_2,\dots$ is a sequence of polynomials in $\C[x]$ where $f_d$ has degree $d$ and $f_d(0)\neq0$. If $L(f_d) = e^{o(d)}$ as $d\to\infty$ then
	\begin{align*}
		\lim_{d\to\infty} \mu_{\{ f_d(x) \} } = \nu_{ \{|z|=1\} }
	\end{align*}
	in the sense of ``weak convergence'' of measures.
\end{thm}

In general, we consider two situations. First, take a partition family $s(n)$ with partition-theoretic statistic $s(m,n)$ such that $s(-m,n) = s(m,n)$ and $s(m,n) \geq 0$ for all $m \in \Z$ and $n\in \N$, with $s(0,n)\neq 0$ for all $n$. Assume that $s(0,n) + 2\sum_{m=1}^{\ell} s(m,n) = s(n)$ for some $\ell \in \N$. Assume that $s(n) \sim e^{o(\ell)}$. Consider the principal polynomial 
\begin{align*}
	S_n(w) \coloneqq \frac{s(0,n)}{2} + \sum_{1\leq m \leq \ell} s(m,n) w^m.
\end{align*}
Then using Theorem \ref{Thm: Granville} it is not difficult to show that as $n\to \infty$ the roots of $S_n(w)$ tend to equidistribution around the unit circle. Of course, the difficulty is usually in finding the asymptotic behaviour of $s(n)$. The families $s(m,n)$ often occur as ranks and cranks of various objects, which are allowed to be negative.

Similarly, we also consider a partition family $t(n)$ with partition-theoretic statistic $t(m,n)$ which vanishes for $m <0$, where $t(m,n) \geq 0$ for all $m \in \N_0$ and $n\in \N$, with $t(0,n)\neq 0$ for all $n$. Assume that $\sum_{m=0}^{\ell} t(m,n) = t(n)$ for some $\ell \in \N$. Then we consider the principal polynomial 
\begin{align}
	T_n(w) \coloneqq \sum_{0\leq m \leq \ell} t(m,n) w^m.
\end{align}
Again, using Theorem \ref{Thm: Granville} we see that as $n\to \infty$ the roots of $T_n(w)$ tend to equidistribution around the unit circle. The families $t(m,n)$ often occur from partition statistics which are naturally positive counts, for example partitions with number of parts equal to $m$.

In  Section \ref{sec_equid}, following these methods, we establish equidistribution results for some specific partition polynomials of interest. Many other partition-theoretic families can be shown to have principal polynomials whose roots tend to equidistribution on the unit circle in a similar way. We pay special attention to strongly unimodal sequence polynomials in Section \ref{sec_sus} and $t$-hook polynomials in Section \ref{sec_thooks}, for which equidistribution-type properties are less clear, and offer various computations and open questions.

\subsection{Equidistribution of roots of partition polynomials}\label{sec_equid}
\subsubsection{Rank and crank polynomials.}
In \cite{BG}, the principal polynomials of the rank and crank of ordinary partitions were each shown to have roots whose arguments tend to equidistribution on the unit circle as $n \to \infty$. The authors of \cite{BG} go on to conjecture that the zero attractor for these polynomials is the unit circle. To the best of the authors' knowledge, this is not confirmed in the literature, and so we record the result here. 

Let $N(m,n)$ and $M(m,n)$ denote the number of partitions of $n$ with rank and crank $m$ respectively. We define the principal polynomials
\begin{align*}
	N_n(w) \coloneqq \frac{N(0,n)}{2} +  \sum_{1\leq m < n} N(m,n) w^m, \qquad M_n(w) \coloneqq \frac{M(0,n)}{2} +  \sum_{1\leq m \leq n} M(m,n) w^m,
\end{align*}
noting that the factor of $\frac{1}{2}$ in the constant term of each is not present in \cite{BG}. In particular, recall that
\begin{align*}
	\sum_{-n < m < n} N(m,n) = \sum_{-n \leq m \leq n} M(m,n) = p(n).
\end{align*}
Moreover, we have that $N(0,n) \neq 0$ and $M(0,n) \neq 0$ for $n\geq 1$, along with the well-known asymptotic formula of Hardy and Ramanujan \cite{HardyRamanujan}
\begin{align*}
	p(n) \sim \frac{1}{4n\sqrt{3}}\exp\left(\pi \sqrt{\frac{2n}{3}}\right), \qquad n \to \infty.
\end{align*}
Thus the assumptions of Theorem \ref{Thm: Granville} apply, and we conclude the following theorem.
\begin{thm}\label{thm: rcequid}
		As $n \to \infty$ the roots of the rank and crank principal polynomials, i.e. $N_n(w)$ and $M_n(w)$, tend to equidistribution on the unit circle. 
\end{thm}

\subsubsection{spt-crank polynomials.}   Recall the $\sptc$ generating function given in \eqref{eqn: spt crank gen fn}.
For fixed $n \in \N$ consider the principal polynomial attached to the $\sptc$ given by
\begin{align*}
	P_n(w) = \tfrac12 N_S(0,n)  + \sum_{1\leq m < n} N_S(m,n) w^m.
\end{align*}
We have that $P_n(0) =  \frac12 N_S(0,n)= \frac12 \operatorname{ospt(n)}$ by Section 2 of \cite{ADR}, where $\operatorname{ospt}$ counts the difference of first moments of the ordinary crank and rank distributions. Moreover, by \cite[Theorem 3]{ACK} we have that $\operatorname{ospt}(n)> 0$ for $n > 0$, and so we need only check that 
\begin{align*}
\frac{ N_S(0,n) + 2 \displaystyle \sum_{j=1}^{n-1} N_S(n-1,n)}{(2N_S(0,n)N_S(n-1,n))^{\frac{1}{2}}} = e^{o(n-1)}
\end{align*}
The numerator is equal to $\operatorname{spt}(n)$ by definition. Using the known asymptotic for $\operatorname{spt}(n)$ given in \cite{Bri}
\begin{align*}
	\operatorname{spt}(n) \sim \frac{1}{2\sqrt{2}\pi \sqrt{n}} e^{\pi \sqrt{\frac{2n}{3}}} = e^{o(n-1)},
\end{align*}
 we may again apply Theorem \ref{Thm: Granville} to conclude that the zeros of $Q_n(w)$ are equidistributed on the unit circle as $n \to \infty$. We record this as a theorem.
\begin{thm}\label{thm: sptequid}
	As $n \to \infty$ the roots of $P_n(w)$ tend to equidistribution on the unit circle. 
\end{thm}

\subsubsection{Unimodal sequence polynomials.}
A  \textit{(weakly) unimodal sequence} of size $n$ is a sequence of positive integers $\{a_j\}_{1\leq j \leq s}$ such that
\begin{align}\label{eqn: unimodal}
a_1 \leq a_2 \leq \dots a_{k-1} \leq  \overline{a}_k \geq a_{k+1} \geq \dots \geq a_s,\qquad \sum_{j=1}^{s} a_j = n.
\end{align}
We follow the notation of \cite{BJM} by indicating with the overline on the peak $a_{k}$ that if the largest part is repeated, the sequences may be further distinguished by specifying the location of the peak. Unimodal sequences are ubiquitous in number theory and in wider mathematics, and we refer the reader to \cite{St} and the references therein for many beautiful examples.

The rank of a unimodal sequence is defined to be the number of parts to the right of the peak minus the number of parts before the peak. If we consider unimodal sequences of size $n$ with rank $m$, denoted by $u(m,n)$, then the generating function is given by (see e.g. \cite{BJM})
\begin{align*}
	U(w;q) \coloneqq \sum_{n\geq0}\sum_{m \in \Z} u(m,n) w^m q^n = \sum_{n \geq 0} \frac{q^n}{(w q;q)_n (w^{-1}q;q)_n}.
\end{align*}

For fixed $n \in \N$, we consider the principal part polynomial defined by
\begin{align*}
	Q_n(w) \coloneqq \tfrac12 u(0,n) + \sum_{1\leq m < n} u(m,n) w^m.
\end{align*}
Then, since $Q_n(0) =  \frac12 u(0,n)  \neq 0$, we need only check that 
\begin{align*}
		\frac{ u(0,n) + 2 \displaystyle \sum_{j=1}^{n-1} u(j,n)}{(2u(0,n)u(n-1,n))^{\frac{1}{2}}} = e^{o(n-1)}
\end{align*}
as $n\to\infty$. Using that $u(-m,n) = u(m,n)$, we see that the numerator counts exactly the number of unimodal sequences of size $n$, denoted by $u(n)$. Using \cite[Theorem 1.1 (1)]{BJM} with $k=0$, which gives
\begin{align*}
	u(n) \sim \frac{1}{8\cdot 3^{\frac{3}{4}} n^{\frac{5}{4}}} e^{\pi \sqrt{\frac{4n}{3}}} = e^{o(n-1)},
\end{align*}
we apply Theorem \ref{Thm: Granville} we conclude that the zeros of $Q_n(w)$ tend toward equidistribution on the unit circle as $n \to \infty$. We record this as a theorem.
\begin{thm}\label{thm: uniequid}
	As $n \to \infty$ the roots of $Q_n(w)$ tend to equidistribution on the unit circle. 
\end{thm}

\subsection{Non-equidistribution of roots of partition polynomials}
In this section we collect some examples of partition polynomials whose roots numerically appear to not be equidistributed on the unit circle as $n$ grows. In each case, we explain why one cannot appeal to Theorem \ref{Thm: Granville}, and offer some numerical data and open questions.

\subsubsection{Strongly unimodal sequences}\label{sec_sus}

A close relative of unimodal sequences are known as strongly unimodal sequences, where in \eqref{eqn: unimodal} one requires the inequalities to be strict. We denote the number of strongly unimodal sequences of size $n$ with rank $m$ by $u^*(m,n)$, and the number of strongly unimodal sequences of size $n$ by $u^*(n)$. Numerical experiments (using SageMath \cite{sage}) suggest that the roots of the principal polynomial for the rank of strongly unimodal sequences do not tend to equidistribution around the unit circle. Figure \ref{Fig: strongly unimodal} plots the roots of the principal polynomial of the rank of strongly unimodal sequences for $n=200$ and $500$.

Let $R_n(w)$ be the principal polynomial attached to strongly unimodal sequences of size $n$. To apply Theorem \ref{Thm: Granville}, one would need that the growth of strongly unimodal sequences of size $n$ is $o(d(n))$, where $d(n)$ is the degree of $R_n(w)$. By \cite{Rhoades} we have
\begin{align*}
	u^*(n) \sim \frac{\sqrt{3}}{2(24n-1)^{\frac{3}{4}}} \exp\left(\frac{\pi}{6}\sqrt{24n-1}\right)
\end{align*}
as $n \to \infty$. However, it is not difficult to see that $d(n)$ is given by taking the index of next triangular number above $n$ and then taking off $1$ (which is very roughly $\sqrt{n}$). To see this, simply write out the strongly unimodal sequences diagrammatically. The largest possible rank will be given by a strictly decreasing triangle starting with the largest size in the first column, and a single dot in the final column. 

Thus Theorem \ref{Thm: Granville} does not apply, and we do not expect the roots of $R_n(w)$ to tend to be equidistributed on the unit disk as $n$ grows (at least, using this method). 

\begin{question}
	What is the limiting distribution of the roots of $R_n(w)$?
\end{question}

\begin{figure}
	\centering
	\begin{subfigure}{.4\textwidth}
		\centering
		\includegraphics[width=0.9\linewidth]{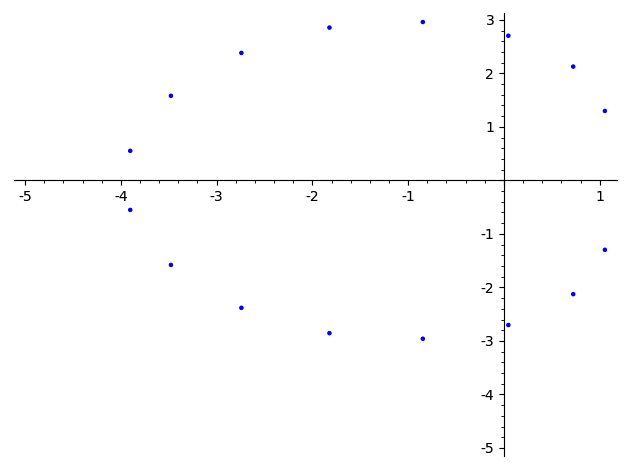}
		\caption{$n=200$}
	\end{subfigure}%
	\begin{subfigure}{.4\textwidth}
		\centering
		\includegraphics[width=0.9\linewidth]{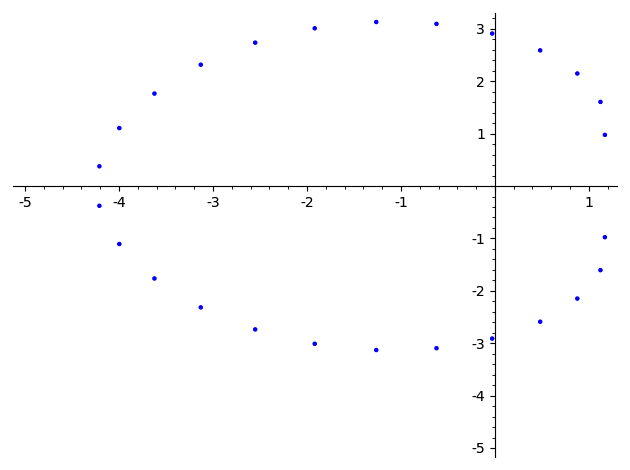}
		\caption{$n=500$}
		\label{$n=500$}
	\end{subfigure}
	\caption{Roots of $R_n(w)$ for two choices of $n$.}
	\label{Fig: strongly unimodal}
\end{figure}

\subsubsection{$t$-hooks}\label{sec_thooks}
Recall the hook length of a partition introduced in Section \ref{Sec: t-core}. The hook lengths which are multiples of a fixed positive integer $t$ are called {\it $t$-hooks}, and been a central object of focus in several recent papers, including work of Bringmann, Craig, Ono, and the second author \cite{BCMO}, who determined an exact formula for the number of $t$-hooks in partitions as well as the non-equidistribution properties over arithmetic progressions.

We also remark that the case of $t=2,3$ appears to be similar numerically. However, the interested researcher aiming to answer this question should be aware that their behaviour may be influenced by the fact that there are arithmetic progressions on which the number of $2$-hooks and $3$-hooks congruent to $a \pmod{b}$ identically vanish. Moreover, the number of $t$-hooks of partitions is not (in general) equidistributed among congruence classes (see \cite{BCMO}).

If we let $\mathcal{H}_t(\lambda)$ denote the multiset of $t$-hooks of a partition $\lambda$, then Han \cite{Han} proved that the generating function for $t$-hooks in partitions 
\begin{equation*}
	H_t(w;q):=\sum_{\lambda \in \mathcal{P}} w^{\# \mathcal{H}_t(\lambda)}q^{|\lambda|} = \sum_{m,n\geq 0} c_t(m,n) w^m q^n
\end{equation*}
takes the following form.

\begin{thm}{\text {\rm (Corollary 5.1 of \cite{Han})}}\label{HanFunction}
	As formal power series, we have
	$$
	H_t(w;q)=\frac{1}{ \prod_{n=1}^{\infty} \left( 1-(w q^t)^n \right)^t}  \prod_{n=1}^{\infty}
	\frac{\left(1-q^{tn}\right)^t}{1-q^n}.
	$$
\end{thm}

We define the principal polynomial for $t$-hooks by
\begin{align*}
	H_{t,n}(w) \coloneqq \sum_{0\leq m \leq n} c_t(m,n) w^m.
\end{align*}
 and below we include several figures on the roots of $H_{t,n}(w)$ for several choices of $t$ and $n$. As can be seen, the zeros appear to tend to lie on two distinct circles as $n$ grows although these circles may have radii that tend to a common value of $1$. In every case, there are several ``sporadic'' zeros lying far from any apparent circles outside of the unit circle. For the cases $t=4,5$ and $n=2000$ (see Figures \ref{Fig: 4-hook} and \ref{Fig: 5-hook}) there appear to be further zeros far from the two apparent circles.

\begin{question}
	What is the limiting distribution of the zeros of $H_{t,n}(w)$?
\end{question}

\begin{figure}
	\centering
	\begin{subfigure}{.4\textwidth}
		\centering
		\includegraphics[width=.5\linewidth]{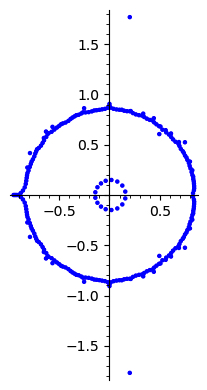}
	\caption{$t=4, n=1000$}
\label{Fig: 4-hook n=1000}
	\end{subfigure}%
	\begin{subfigure}{.4\textwidth}
		\centering
		\includegraphics[width=.5\linewidth]{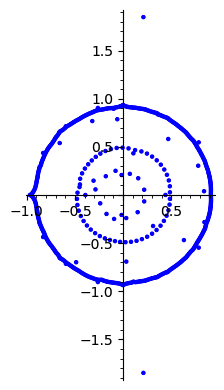}
	\caption{$t=4, n=2000$}
		\label{Fig: 4-hook n=2000}
	\end{subfigure}
	\caption{Roots of the $4$-hook principal polynomial.}
	\label{Fig: 4-hook}
\end{figure}

\begin{figure}
	\centering
	\begin{subfigure}{.4\textwidth}
		\centering
		\includegraphics[width=.45\linewidth]{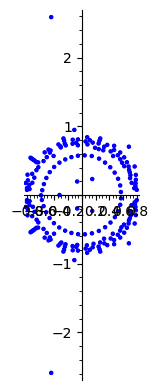}
	\caption{$t=5, n=1000$}
		\label{Fig: 5-hook n=1000}
	\end{subfigure}%
	\begin{subfigure}{.4\textwidth}
		\centering
		\includegraphics[width=.4\linewidth]{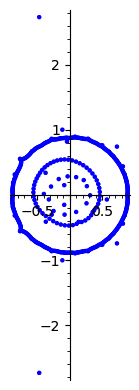}
	\caption{$t=5, n=2000$}
		\label{Fig: 5-hook n=2000}
	\end{subfigure}
	\caption{Roots of the $5$-hook principal polynomial.}
	\label{Fig: 5-hook}
\end{figure}

\begin{figure}
	\centering
	\begin{subfigure}{.4\textwidth}
		\centering
		\includegraphics[width=.5\linewidth]{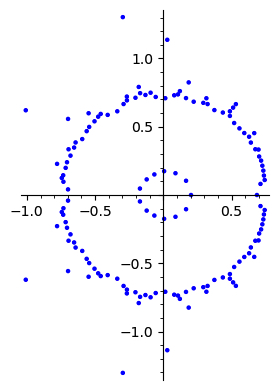}
		\caption{$t=7, n=1000$}
		\label{Fig: 7-hook n=1000}
	\end{subfigure}%
	\begin{subfigure}{.4\textwidth}
		\centering
		\includegraphics[width=.8\linewidth]{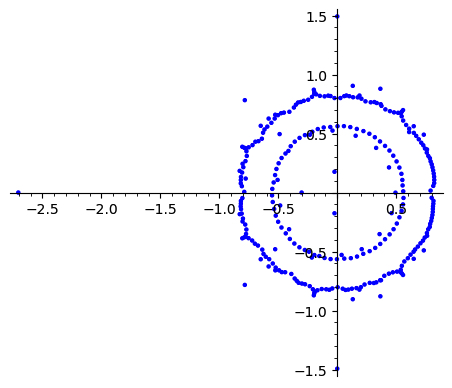}
		\caption{$t=7, n=2000$}
		\label{Fig: 7-hook n=2000}
	\end{subfigure}
	\caption{Roots of the $7$-hook principal polynomial.}
	\label{Fig: 7-hook}
\end{figure}

\FloatBarrier

\end{document}